\documentclass[11pt]{article}

\usepackage{amsmath}
\usepackage{amsfonts}
\usepackage{amssymb}
\usepackage{amsbsy}
\usepackage{amsthm}
\usepackage{fullpage}

 \numberwithin{equation}{section}

 \newtheorem{theorem}{Theorem}[section]

  \newtheorem{lemma}[theorem]{Lemma}

\newtheorem{definition}[theorem]{Definition}

\newcommand{\EE}{\mathbb{E}}
\newcommand{\VV}{\mathbb{V}}

\newcommand{\eps}{\varepsilon}

\newcommand{\dx}{\, {\rm d} x}

\newcommand{\R}{\, {\mathbb R}}

\begin{document}

%
%

\title{Multilevel Monte Carlo methods for highly heterogeneous media}

\author{Aretha L. Teckentrup \\ [12pt]
University of Bath\\
        Claverton Down\\
Bath, BA2 7AY, UK\\
}
\date{}
\maketitle

\begin{abstract}
We discuss the application of multilevel Monte Carlo methods to elliptic partial differential equations with random coefficients. Such problems arise, for
example, in uncertainty quantification in subsurface flow modeling. We give a brief review of recent advances in the numerical analysis of the multilevel algorithm under minimal assumptions on the random coefficient, and extend the analysis to cover also tensor--valued coefficients, as well as point evaluations. Our analysis includes as an example log--normal random coefficients, which are frequently used in applications.
\end{abstract}

\section{Introduction}
\label{sec:intro}
There are many situations in which modeling and computer simulation are indispensable
tools and where the mathematical models employed have been demonstrated to give adequate
representations of reality.  However, the parameters appearing in the models often have to be
estimated from measurements and are, therefore, subject to uncertainty. This uncertainty
propagates through the simulations and quantifying its impact on the results is frequently
of great importance.

A good example is provided by the problem of assessing
the safety of a potential deep geological repository for radioactive wastes. Any radionuclides
leaking from such a repository could be transported back to the human environment by
groundwater flowing through the rocks beneath the earth's surface. The very
long timescales involved mean that modeling and simulation are essential in
evaluating repository performance. The study of groundwater flow is well established,
and there is general scientific consensus that in many situations Darcy's Law can be expected to lead to an
accurate description of the flow \cite{demarsily}. The main parameter appearing in Darcy's
Law is the permeability, which characterizes how easily water can flow through the rock
under a given pressure gradient. In practice it is only possible to measure the
permeability at a limited number of spatial locations, but it is required at all points of the computational
domain for the simulation.  This fact is the primary source of
uncertainty in groundwater flow calculations. Understanding and quantifying the impact of this uncertainty
on predictions of radionuclide transport is essential for reliable repository safety assessments.

A widely used approach for dealing with uncertainty in groundwater flow is to represent the permeability
as a random field \cite{delhomme,demarsilyetal}. A model frequently used is a log--normal random field, with a covariance function that is only Lipschitz continuous. Individual realizations of such fields have low spatial regularity and significant spatial variation,
making the problem of solving for the pressure very costly. The notoriously slow rate of convergence of the standard Monte Carlo algorithm means that many such realizations are required to obtain accurate results, rendering the problem computationally unfeasible. 

In this paper, we therefore employ the multilevel Monte Carlo (MLMC) method. This method was first introduced by \cite{giles1} in the context of stochastic differential equations in finance, and similar ideas were also used by \cite{heinrich} and \cite{brandt2}. In the context of our groundwater flow model problem, it was shown in for example \cite{cgst11} and \cite{tsgu12}, that the multilevel method leads to a significant reduction in the computational cost required to achieve a given accuracy.

The main challenge in the numerical analysis of MLMC methods for elliptic partial differential equations (PDEs) with random coefficients, is the quantification of the numerical discretization error, or in other words the {\em bias} of the estimator. Models for the random coefficient frequently used in applications, such as log--normal random fields, are not uniformly coercive, making the numerical analysis challenging. A rigorous analysis of the MLMC algorithm under minimal assumptions on the random coefficient was recently carried out by \cite{cst11} and \cite{tsgu12}. In particular, uniform coercivity or boundedness were not assumed in these papers. If one does assume uniform coercivity and boundedness of the coefficient, the analysis of the discretization error is classical, and an analysis of the MLMC method for this case can be found in \cite{bsz11}. Other related works on numerical errors for elliptic PDEs with random coefficients are \cite{charrier} and \cite{gittelson}.

The aim of this paper is to extend the theory in \cite{tsgu12}. We here consider the case of more general, tensor--valued models of the permeability, which are often used in applications to model orthotropic media. We will also prove convergence of the MLMC algorithm for point evaluations of the pressure or the Darcy flux.

The outline of the paper is as follows. In \S \ref{sec:multi}, we describe the multilevel Monte Carlo algorithm applied to elliptic PDEs with random coefficients, and discuss its performance. In \S \ref{sec:theory}, we then prove an upper bound on the computational cost of the multilevel Monte Carlo estimator. We recall some of the main results from (\cite{cst11,tsgu12}), before extending the results to tensor--coefficients and point evaluations.

\section{Multilevel Monte Carlo Simulation}
\label{sec:multi}
The classical equations governing a steady state, single phase flow, are Darcy's law coupled with an incompressibility condition. These equations can be written in second order form as 
\begin{equation}
\label{mod}
-\mathrm{div} \left(\mathbf A \nabla u\right) = 
f, \ \qquad \text{in} \quad D \subset \mathbb R^d, 
\end{equation}
subject to appropriate boundary conditions. Here, $\mathbf A$ is the permeability tensor, $u$ is the resulting pressure field, and $f$ are the source terms. Modeling $\mathbf A$ as a random field, $u$ also becomes a random field.

In applications, one is then usually interested in finding the expected value of some functional $Q=M(u)$ of the solution $u$ to our model problem~\eqref{mod}. This could for example be the value of the pressure $u$ or the Darcy flux $- \mathbf A \nabla u$ at or around a given point in the computational domain, or outflow over parts of the boundary. Since $u$ is not easily accessible, $Q$ is often approximated by the quantity $Q_h:=M(u_h)$, where $u_h$ is a finite dimensional approximation to $u$, such as the finite element solution on a sufficiently fine spatial grid $\mathcal{T}_h$. 

To estimate $\EE\left[Q\right]$,
we then compute approximations (or {\it estimators}) $\widehat{Q}_h$ to
$\EE\left[Q_h\right]$, and quantify the accuracy of our approximations via the root
mean square error (RMSE)
\[
e(\widehat{Q}_h) := \left(\mathbb{E}\big[(\widehat{Q}_h - \mathbb{E}(Q))^2\big]\right)^{1/2}.
\]
The computational cost $\mathcal{C}_\eps(\widehat{Q}_h)$ of our estimator is then
quantified by the number of floating point operations that are needed to achieve a
RMSE of $e(\widehat{Q}_h) \leq \eps$. This will be referred to as the $\eps$--cost.

The classical Monte Carlo (MC) estimator for $\EE\left[Q_h\right]$ is
\begin{equation*}
\label{MC}
\widehat{Q}^\mathrm{MC}_{h,N} := \frac{1}{N} \sum_{i=1}^N Q_h(\omega^{(i)}),
\end{equation*}
where $Q_h(\omega^{(i)})$
is the $i$th sample of $Q_h$ and $N$ independent samples are computed in total.

There are two sources of error in the estimator~\eqref{MC}, the
approximation of $Q$ by $Q_h$, which is related to the spatial
discretisation, and the sampling error due to replacing the expected value
by a finite sample average. This becomes clear when expanding the mean
square error (MSE) and using the fact that for Monte Carlo
$\EE[\widehat{Q}^\mathrm{MC}_{h,N}] = \EE[Q_h]$ and
$\VV[\widehat{Q}^\mathrm{MC}_{h,N}] = N^{-1} \, \VV[Q_h],$ where $\VV[X] :=
\EE[(X-\EE[X])^2]$ denotes the variance of the random variable $X:\Omega \to \mathbb{R}$.
We get
\begin{equation}
\label{msesd2}
e(\widehat{Q}^\mathrm{MC}_{h,N})^2 \ = \ N^{-1} \VV[Q_h] + \big(\EE[Q_h - Q]\big)^2.
\end{equation}
A sufficient condition to achieve a RMSE of $\eps$ with this estimator is that both of
these terms are less than $\eps^2/2$. For the first term, this is achieved by choosing
a large enough number of samples, $N=\mathcal{O}(\eps^{-2})$. For the second term, we
need to choose a fine enough finite element mesh $\mathcal{T}_h$, such that
$\EE[Q_h - Q] = \mathcal{O}(\eps)$.

The main idea of the MLMC estimator is very simple. We sample not just from one
approximation $Q_h$ of $Q$, but from several. Linearity of the expectation operator
implies that
\begin{equation*}
\EE[Q_h] = \EE[Q_{h_0}] + \sum_{\ell=1}^L \EE[Q_{h_\ell} - Q_{h_{\ell-1}}]
\label{eq:identity}
\end{equation*}
where $\{h_\ell\}_{\ell = 0, \dots, L}$ are the mesh widths of a sequence of increasingly
fine triangulations $\mathcal{T}_{h_\ell}$ with
$\mathcal{T}_h := \mathcal{T}_{h_L}$, the finest mesh. 
Hence, the expectation on the finest mesh is equal to
the expectation on the coarsest mesh, plus a sum of corrections adding
the difference in expectation between simulations on consecutive meshes.
The multilevel idea is now to independently estimate each of these
terms such that the overall variance is minimized for
a fixed computational cost.

Setting for convenience $Y_0 := Q_{h_0}$ and $Y_\ell := Q_{h_\ell} - Q_{h_{\ell-1}}$, for
$1 \leq \ell \leq L$, we define the MLMC estimator simply as
\begin{equation*}
\widehat{Q}^\mathrm{ML}_{h,\{N_\ell\}} \ := \ \sum_{\ell=0}^L \widehat{Y}^\mathrm{MC}_{\ell, N_\ell}, 
\end{equation*}
where $\widehat{Y}^\mathrm{MC}_{\ell, N_\ell}$ is the standard MC estimator for $Y_\ell$,
\begin{equation*}
\widehat{Y}^\mathrm{MC}_{\ell, N_\ell} 
\ = \  \frac{1}{N_{\ell}} \sum_{i=1}^{N_{\ell}} Y_\ell(\omega^{(i)}).
\end{equation*}
Here, it is important to note that $Y_\ell(\omega^{(i)}) = Q_{h_\ell}(\omega^{(i)}) -
Q_{h_{\ell-1}}(\omega^{(i)})$, i.e. the quantity $Y_\ell(\omega^{(i)})$ is computed using the same sample on both meshes.

Since all the expectations $\EE[Y_{\ell}]$ are estimated independently in
\eqref{eq:identity}, the variance of the MLMC estimator is
$
\sum_{\ell=0}^L N_{\ell}^{-1}\,\VV[Y_\ell]
$
 and expanding as in~(\ref{msesd2}) leads again to a MSE of the form 
\begin{equation*}\label{mseml}
e(\widehat{Q}^\mathrm{ML}_{h,\{N_\ell\}})^2 \; := \;
\mathbb{E}\Big[\big(\widehat{Q}^\mathrm{ML}_{h,\{N_\ell\}} - \mathbb{E}[Q]\big)^2\Big] \; = \;
\sum_{\ell=0}^L N_{\ell}^{-1}\,\VV[Y_{\ell}] \;+ \; \big(\mathbb{E}[Q_{h} - Q]\big)^2.
\end{equation*}
As in the classical MC case before, we see that the MSE consists of two terms,
the variance of the estimator and the error in mean between $Q$ and $Q_h$. Note that
the second term is identical to the second term for the classical MC method in
\eqref{msesd2}. 
A sufficient condition to achieve a RMSE of $\eps$ is again to make
both terms less than $\eps^2/2$. This is easier to achieve with the MLMC estimator,
as
\begin{itemize}
\item for sufficiently large $h_0$, samples of $Q_{h_0}$ are much cheaper to obtain
than samples of $Q_{h}$;
\item the variance $Y_{\ell}$ tends to 0 as $h_\ell \to 0$, meaning we
need fewer samples on $\mathcal{T}_{h_\ell}$, for $\ell > 0$.
\end{itemize}

Let now $\mathcal{C}_\ell$ denote the cost to obtain one sample of $Q_{h_\ell}$. Then
we have the following results on the  $\eps$--cost of the MLMC estimator
(cf.~\cite{cgst11,giles1}).

\begin{theorem}
\label{main_thm}
Suppose 
there are positive constants $\alpha, \beta, \gamma, c_{\scriptscriptstyle \mathrm{M1}}, c_{\scriptscriptstyle \mathrm{M2}}, c_{\scriptscriptstyle \mathrm{M3}} > 0$ such that
$\alpha\!\geq\!\frac{1}{2}\,\min(\beta,\gamma)$ and
\begin{itemize}
\item[{\bf M1.}]
$\displaystyle
\left| \EE[Q_{h} - Q] \right| \  \leq c_{\scriptscriptstyle \mathrm{M1}} \ h^{\alpha},
$
\item[{\bf M2.}]
$\displaystyle
\VV[Q_{h_\ell}-Q_{h_{\ell-1}}] \ \leq c_{\scriptscriptstyle \mathrm{M2}} \ h_\ell^{\beta},
$
\item[{\bf M3.}]
$\displaystyle
\mathcal{C}_{\ell} \ \leq c_{\scriptscriptstyle \mathrm{M3}} \ h_{\ell}^{-\gamma},
$
\end{itemize}
Then, for any $\;\eps < e^{-1}$, there exist an $L$ and a sequence
$\{N_{\ell}\}_{\ell=0}^L$, such that $e(\widehat{Q}^\mathrm{ML}_{h, \{N_\ell\}}) < \eps$ and
\[
\mathcal{C}_\eps(\widehat{Q}^\mathrm{ML}_{h, \{N_\ell\}}) \ \lesssim \ \left\{\begin{array}{ll}
\ \eps^{-2}              ,    & \text{if } \ \beta>\gamma, \\[0.02in]
\ \eps^{-2} (\log \eps)^2,    & \text{if } \ \beta=\gamma, \\[0.04in]
\ \eps^{-2-(\gamma - \beta)/\alpha}, & \text{if } \ \beta<\gamma,
\end{array}\right.
\]
where the hidden constant depends on $c_{\scriptscriptstyle \mathrm{M1}}, c_{\scriptscriptstyle \mathrm{M2}}$ and $c_{\scriptscriptstyle \mathrm{M3}}$. 
For the classical MC estimator we have \\ $\mathcal{C}_\eps(\widehat{Q}^\mathrm{MC}_h) \, \lesssim \, \eps^{-2-\gamma/\alpha}$, where the hidden constant depends on $c_{\scriptscriptstyle \mathrm{M1}}$ and $c_{\scriptscriptstyle \mathrm{M3}}$.
\end{theorem}

The convergence rates $\alpha$ and $\beta$ in Theorem\ref{main_thm} are related to the convergence of the spatial discretization error, and have been proven for various quantities of interest in \cite{cst11} and \cite{tsgu12}. In \S \ref{sec:linfty}, we further extend this theory to point evaluations of the pressure and the flux. Typical values of $\alpha$ and $\beta$ for the model problem considered in this paper are $\alpha=1$ and $\beta=2$ for rough models of the permeability and $\alpha=2$ and $\beta=4$ for smoother models (this is made more precise in \S \ref{sec:theory}). The rate $\gamma$ is related to the cost of numerically solving the PDE for one realization of the random coefficient. This involves producing a sample of the random coefficient, and solving a linear system of equations. The cost of solving the linear system will generally be dominant, and with an optimal linear solver, the cost of one such solve is proportional to $h_\ell^{-d}$, the number of unknowns, and so $\gamma \approx d$.

In Table~\ref{tab:theory}, we show the $\eps$--costs as predicted by Theorem~\ref{main_thm}, for typical values of $\alpha$ and $\beta$. We assume an almost optimal linear solver, and take $\gamma$ to be slightly larger than $d$. We see that the gains we can expect from using the MLMC estimators are always significant, usually in the order of two orders of magnitude. It is also worth noting that although the actual $\eps$--costs are higher in the case of the rough model problem with $\alpha=1$ and $\beta=2$, the gains we can expect from MLMC are also greater in this case. 

\begin{table}[h]
\begin{center}
\caption{\label{tab:theory} Upper bounds for the $\eps$-costs of classical 
and multilevel Monte Carlo from Theorem~\ref{main_thm} in the cases $\alpha=1, \beta=2$ (left) and $\alpha=2, \beta=4$ (right), with $\gamma=d+\delta$ in both cases, where $\delta >0$ is a small constant. $d$ is the spatial dimension from~\eqref{mod}.}
\begin{tabular}{|c|cc|cc|}
\hline
& \multicolumn{2}{c|}{$\alpha=1$, $\beta=2$} & \multicolumn{2}{c|}{$\alpha=2$, $\beta=4$} \\
$d$ & ~~MC~~ & MLMC & ~~MC~~ & MLMC \\\hline
$1$ & $\eps^{-3}$ & $\eps^{-2}$ & $\eps^{-5/2}$ & $\eps^{-2}$ \\
$2$ & $\eps^{-4}$ & $\eps^{-2}$ & $\eps^{-3}$ & $\eps^{-2}$ \\
$3$ & $\eps^{-5}$ & $\eps^{-3}$ & $\eps^{-7/2}$ & $\eps^{-2}$ \\\hline
\end{tabular} 
\end{center}
\end{table}

The reduction in cost associated with the MLMC estimator over standard MC is largely due to the fact that the number of samples needed on the finer grids is greatly reduced. Most of the uncertainty can already be captured on the coarse grids, and so the MLMC estimator shifts some of the computational effort on to the coarse grids. Exactly how much of the computational effort can (and should) be shifted towards the coarse grids, depends on the model problem and the quantity of interest $Q$. The MLMC algorithm described above chooses the number of samples on each level in such a way that the computational cost of the estimator is minimized, subject to the overall variance of the estimator being less than $\eps^2/2$. This can lead to three different scenarios: the computational cost could be predominantly on the coarse levels, spread evenly across the levels, or predominantly on the fine levels. This corresponds to the three upper bounds given in Theorem~\ref{main_thm} above.

To make this more precise, note that for given $\{N_\ell\}$ and $\{\mathcal C_\ell\}$, the computational cost of the MLMC estimator is 
\[
\mathcal{C}(\widehat{Q}^\mathrm{ML}_{h, \{N_\ell\}}) = \sum_{\ell=0}^L N_\ell \, \mathcal C_\ell.
\]
Treating the $N_{\ell}$ as continuous variables, the cost of the MLMC estimator is minimized
for a fixed variance by choosing
\begin{equation*}
\label{choice_Nl}
N_{\ell} \ \eqsim \ \sqrt{\VV[Y_{\ell}] / \mathcal{C}_{\ell}}\, ,
\end{equation*}
with the constant of proportionality chosen so that the overall variance is
$\eps^2/2$. The total cost on level $\ell$ is then proportional to
$\sqrt{\VV[Y_{\ell}] \, \mathcal{C}_{\ell}}$, and hence
\[
\mathcal{C}(\widehat{Q}^\mathrm{ML}_{h, \{N_\ell\}}) \ \lesssim \
\sum_{\ell=0}^L \sqrt{\VV[Y_{\ell}] \, \mathcal{C}_{\ell}}.
\]
If the variance $\VV[Y_{\ell}]$ decays faster with $\ell$ than the cost
$\mathcal{C}_{\ell}$ increases, i.e. if $\beta > \gamma$, the dominant term will be on level~$0$. Similarly, if $\VV[Y_{\ell}]$ decays slower than  $\mathcal{C}_{\ell}$ increases, the dominant term will be on the finest level $L$, and if $\VV[Y_{\ell}]$ decreases at the same rate as $\mathcal{C}_{\ell}$ increases, the cost is spread evenly across all levels. In the context of our model problem in subsurface flow, we are usually in the last regime, where $\beta < \gamma$. Especially in 2 or 3 space dimensions, the cost of obtaining one sample grows very rapidly, and the dominant cost will always be on the finest level. It is worth to note that if $\beta=2\alpha$, and like here, $\beta < \gamma$, the cost of the MLMC estimator is of the order $\eps^{-\gamma/\alpha}$. This is in fact the same cost as taking only one sample on the finest grid, since we have to choose $h \eqsim \eps^{\alpha}$ to get a MSE of $\mathcal O(\eps^2)$, and the cost of one solve is then $\mathcal C \lesssim h^{-\gamma} = \eps^{-\gamma/\alpha}$, by assumption M3. This means that asymptotically our multilevel Monte Carlo method for the stochastic problem has the same complexity as a deterministic solver for one realization of the same problem.

Another issue which influences the cost of the MLMC estimator, is the choice of the coarsest mesh size $h_0$. The bigger $h_0$ is, the more levels we can include in the MLMC estimator, and the bigger the potential gains are with respect to standard MC. Although the choice of $h_0$ does not influence the asymptotic bounds on the cost given in Theorem~\ref{main_thm}, the choice of $h_0$ does have an effect on the absolute cost of the MLMC estimator for any fixed accuracy $\eps$. In practical applications, $h_0$ must often be chosen to give a minimal level of resolution to the problem in order to get the MLMC estimator with the smallest absolute cost. For the model problem in subsurface flow, where the permeability varies on a very fine scale and is highly oscillatory, very coarse meshes do not yield a good representation of the problem, and including them in the MLMC estimator can lead to a larger absolute cost than necessary. 
One way to circumvent this problem, is to use smoother representations of the permeability on the coarse levels. It was shown in \cite{tsgu12} that, without introducing any additional bias in the MLMC estimator, this strategy allows for the inclusion of much coarser levels, and hence gives a significantly lower absolute cost of the MLMC estimator, even in the context of short correlation lengths.

The rest of the paper is devoted to proving theoretical convergence rates, and thus justifying assumptions M1 and M2 in Theorem~\ref{main_thm}.

\section{Numerical Analysis}
\label{sec:theory}
For simplicity, we consider a particular instance of model problem \eqref{mod}, posed on a Lipschitz--polygonal domain $D \subset \mathbb R^2$ and with homogeneous Dirichlet conditions: Given a probability space $\left( \Omega, \mathcal{A}, \mathbb{P}\right)$ and $\omega \in \Omega$, find $u:\Omega\times D \to \mathbb{R}$ such that
\begin{align}
\label{mod1}
-\mathrm{div} \left(\mathbf A(\omega, x) \nabla u(\omega, x)\right) &= 
f(\omega,x), \ \qquad \mathrm{for} \ x \in D, \\
u(\omega,x) &= 0, \qquad \mathrm{for} \ x \in \Gamma_j\,.  \nonumber
\end{align}
The differential operators $\mathrm{div}$ and $\nabla$ are with respect to $x \in D$, and $\Gamma:= \cup_{j=1}^m \overline \Gamma_j$ denotes the boundary of $D$, partitioned into straight line segments. 
Note that due to the tensor--valued coefficient $\mathbf A(\omega, x)$, this problem is more general than those studied in our earlier papers \cite{cst11} and \cite{tsgu12}. It is of course possible to consider other boundary conditions and/or higher spatial dimensions, and this is done for example in \cite{tsgu12}. One can also include lower order terms in the differential operator.

We will carry out a finite element error analysis of \eqref{mod1}, under minimal assumptions on the coefficient tensor $\mathbf A$ and on the source term $f$. In particular, we do not assume that $\mathbf A$ is coercive and bounded {\em uniformly} in $\omega$, since this is not the case for example for log--normal random fields, which can take values arbitrarily close to zero or infinity for any given realization. The crucial observation is that for each {\em fixed} $\omega$, we have a uniformly coercive and bounded problem (in $x$). The first step in our error analysis is therefore to derive an estimate on the finite element error for a fixed $\omega$. However, in order to be able to compute moments (expectations) of the error, it is crucial that we keep track of how all the constants that appear in our estimates depend on $\omega$, or in other words on $\mathbf A$ and $f$.

The coefficient tensor $\mathbf A(\omega,\cdot)$ is assumed to take values in the space of real--valued, symmetric $d \times d$ matrices. Given the usual norm $|v| := \left(\sum_{i=1}^d |v_i^2|\right)^{1/2}$ on $\mathbb R^d$, we choose the norm on $\mathbb R^{d \times d}$ as the norm induced by $|\cdot|$, or any matrix norm equivalent to it. 

For all $\omega \in \Omega$, let now $\mathbf A_{\mathrm{min}}(\omega)$ be such that 
\begin{equation*}\label{amin}
\mathbf A(\omega,x) \xi \cdot \xi \gtrsim \mathbf A_{\mathrm{min}}(\omega) |\xi|^2, \quad \forall \xi \in \mathbb R^d, \quad \text{uniformly in } x \in D,
\end{equation*}
and define
\begin{equation*}\label{amax}
\mathbf A_{\mathrm{max}}(\omega) := \|\mathbf A(\omega,\cdot)\|_{C(\overline D, \R^{d \times d})}.\end{equation*}
\hspace*{1ex} \\
We make the following assumptions on the input data:
\begin{itemize}
\item[{\bf A1.}] $\mathbf A_{\mathrm{min}} \ge 0$ almost surely and 
$1/\mathbf A_{\mathrm{min}} \in L^{p}(\Omega)$, for all $p \in (0,\infty)$.
\item[{\bf A2.}] $\mathbf A_{i,j} \in C\,^{t}(\overline D), i,j=1,\dots,d$, and $\mathbf A \in L^{p}(\Omega,C\,^{t}(\overline D, \R^{d \times d})))$, for some $0 < t \le 1$ and 
for all $p \in (0,\infty)$.
\item[{\bf A3.}] $f \in L^{p_*}(\Omega,H^{t-1}(D))$
, for some $p_* \in (0,\infty]$.
\end{itemize}

Here, the space $C\,^{t}(\overline D, \R^{d \times d})$ is the space of $d \times d$ matrix--valued, H\"older--continuous functions with exponent $t$, $H^s(D)$ is the usual fractional order Sobolev space, and  $L^q(\Omega, \mathcal B)$ denotes the space of $\mathcal B$-valued random fields, for which the $q^\mathrm{th}$ moment (with respect to the measure $\mathbb P$) of the $\mathcal B$--norm is finite, see e.g \cite{cst11}. A space which will appear in the error analysis later is the space $L^q(\Omega, H^1_0(D))$, which denotes the space of $H_0^1(D)$--valued random fields with the norm on $H_0^1(D)$ being the usual $H^1(D)$--seminorm $|\cdot|_{H^1(D)}$. It is possible to weaken Assumptions A1 and A2 to $1/\mathbf A_{\mathrm{min}}$ and $\|\mathbf A\|_{C\, ^{t}(\overline D, \R^{d \times d})}$ having finite moments of order $p_a$, for some $p_a \in (0,\infty)$, but we will not do this here for ease of presentation. 

An example of a random tensor $\mathbf A(\omega,x)$ that satisfies Assumptions A1 and A2, for all $p \in (0,\infty)$, is a tensor of the form $\mathbf A = \exp(g_1) K_1 + \exp(g_2) K_2$, where $g_1$ and $g_2$ are real--valued Gaussian random fields with a H\"older--continuous mean and a Lipschitz continuous covariance function, and $K_1$ and $K_2$ are deterministic tensors satisfying (deterministic versions of) assumptions A1--A2. For example, $g_i, i=1,2,$ could have constant mean and an exponential covariance function, given by
\begin{equation*}\label{cov:exp}
\mathbb E\Big[(g_i(\omega,x)-\mathbb E[g_i(\omega,x)])(g_i(\omega,y)-\mathbb 
E[g_i(\omega,y)]) \Big]= \sigma^2 \exp(-\|x-y\|/ \lambda)
\end{equation*}
where $\sigma^2$ and $\lambda$ are real parameters known as the {\em variance} and {\em correlation length}, and $\|\cdot \|$ denotes a norm on $\mathbb{R}^d$.
It follows from the results in \cite{cst11} that the resulting random tensor satisfies assumptions A1--A2, for any $t < 1/2$. If we instead choose a smoother covariance function, like the Gaussian covariance
\begin{equation*}\label{cov:gauss}
\mathbb E\Big[(g_i(\omega,x)-\mathbb E[g_i(\omega,x)])(g_i(\omega,y)-\mathbb 
E[g_i(\omega,y)]) \Big]= \sigma^2 \exp(-\|x-y\|^2/ \lambda^2)
\end{equation*}
the resulting random tensor $\mathbf A$ satisfies assumptions A1--A2 with $t=1$.

To simplify the notation in the following, let $0 < C_{\mathbf A,f} < \infty$ denote a generic constant which depends algebraically on $L^q(\Omega)$--norms of $\mathbf A_{\mathrm{max}}, \, 1/\mathbf A_{\mathrm{min}}, \, \|\mathbf A\|_{C\,^{t}(\overline D, \R^{d \times d})}$ and $ \|f\|_{H^{t-1}(D)}$, with $q < p*$ in the case of $\|f\|_{H^{t-1}(D)}$. We will also use the notation $b \lesssim c$ for two positive quantities $b$ and $c$, if $b/c$ is uniformly bounded by a constant independent of $\mathbf A$, $f$ and $h$. 

We will study the PDE~\eqref{mod1} in weak (or 
variational) form, for fixed $\omega \in \Omega$. This is not possible uniformly 
in $\Omega$, but almost surely. In the following we will not explicitly write this 
each time. With $f(\omega,\cdot) \in H^{t-1}(D)$ and 
$0 < \mathbf A_{\mathrm{min}}(\omega) \leq \mathbf A_{\mathrm{max}}(\omega) < \infty$, for 
all $x \in D$, the variational formulation of \eqref{mod1}, parametrized by
$\omega \in \Omega$, is 
\begin{equation}\label{weak}
b_\omega\big(u(\omega,\cdot),v\big) = L_\omega(v)\,, 
\quad \text{for all} \quad v \in H^1_0(D),
\end{equation}
where the bilinear form $b_\omega$ and the linear functional
$L_\omega$ (both parametrized by $\omega \in \Omega$) are defined as usual, for all 
$u,v \in H^1_0(D)$, by
\begin{equation*}\label{bilinear}
b_\omega(u,v) := \int_D \mathbf A(\omega,x)\nabla u(x) \cdot \nabla v(x) \dx \quad \text{and} \quad 
L_\omega(v) := \langle f(\omega,\cdot), v\rangle_{H^{t-1}(D),H^{1-t}_0(D)}\,.
\end{equation*}
We say that for any $\omega \in \Omega$, $u(\omega,\cdot)$ is a 
weak solution of \eqref{mod1} iff $u(\omega,\cdot) \in H^1_0(D)$ and satisfies 
(\ref{weak}).
The following result is classical. It is based on the Lax-Milgram Lemma (cf \cite{hackbusch}).

\begin{lemma}\label{laxm} For almost all $\omega \in \Omega$, the bilinear form 
$b_\omega(u,v)$ is bounded and coercive in $H^1_0(D)$ with respect to $|\cdot|_{H^1(D)}$, 
with constants $\mathbf A_{\mathrm{max}}(\omega)$ and $\mathbf A_{\mathrm{min}}(\omega)$, respectively. 
Moreover, there exists a unique solution $u(\omega,\cdot) \in H^1_0(D)$ to the 
variational 
problem \eqref{weak} and
\[
|u(\omega,\cdot)|_{H^1(D)} \lesssim \frac{\|f(\omega,\cdot)\|_{H^{-1}(D)}}{\mathbf A_{\mathrm{min}}(\omega)}.
\]
\end{lemma}

We now consider finite element approximations of our model problem~\eqref{mod1} using standard, continuous, piecewise linear finite elements. This is not the only possible choice, and the MLMC estimator works equally well with other spatial discretizations. See for example \cite{cgst11} for results with finite volume discretizations, and \cite{gsu12} for an error analysis in the case of mixed finite elements. 

Denote by $\{\mathcal{T}_h\}_{h>0}$ a shape-regular family of simplicial 
triangulations of the domain $D$, parametrized by its mesh width 
$h := \max_{\tau \in \mathcal{T}_h} \mathrm{diam}(\tau)$.

Associated with each triangulation $\mathcal{T}_h$ we define the space
\begin{equation*}\label{fespace}
V_h := \left\{v_h \in C(\overline D) : v_h|_\tau \ \text{linear, for all} \ \tau \in 
\mathcal{T}_h , \ \ \mathrm{and} \ \ 
v_h|_{\Gamma} = 0 \right\}
\end{equation*}
of continuous, piecewise linear functions on $D$ 
that vanish on the boundary.


The finite element approximation of $u$ in $V_{h}$, denoted by $u_h$, is now found by solving 
\[
b_\omega\big(u_h(\omega,\cdot),v\big) = L_\omega(v) \,, 
\quad \text{for all} \quad v \in V_{h},
\]


The key tools in proving convergence of the finite element method are Cea's lemma and a best approximation result (cf \cite{hackbusch,brenner_scott}):

\begin{lemma}[Cea's Lemma] \label{cea} Let Assumptions A1--A3 hold. Then, for almost 
all $\omega \in \Omega$,
\[ 
|(u-u_h)(\omega,\cdot)|_{H^1(D)} \leq \bigg(\frac{\mathbf A_{\mathrm{max}}(\omega)}{\mathbf A_{\mathrm{min}}(\omega)}\bigg)^{1/2} \, \inf_{v_h \in V_h} 
|u(\omega,\cdot)-v_h|_{H^1(D)}.
\]
\end{lemma}

\begin{lemma}\label{infh} 
Let $v \in H^{1+s}(D) \cap H_0^{1}(D)$, for some $0 < s \le 1$.
Then
\begin{equation*}
\inf_{v_h \in V_h} |v-v_h|_{H^1(D)} \lesssim \,\|v\|_{H^{1+s}(D)}\, h^s
\end{equation*}
where the hidden constant is independent of $v$ and $h$. 
\end{lemma}

In order to conclude on the the convergence of $u$ to $u_h$ in the $L^p(\Omega, H^1_0(D))$--norm, or in other words on the convergence of moments of the $H^1(D)$--seminorm of the error, it is hence crucial that we can bound moments of $\|u\|_{H^{1+s}(D)}$, for some $0 < s \leq 1$.  The spatial regularity of $u$ depends both on the regularity of $\mathbf A$ and $f$, and on the geometry of the domain $D$, so we need the following definition in addition to assumptions A1--A3.

\begin{definition}
\label{def:laplace}\em
Let $0 < \lambda_{\Delta}(D) \le 1$ be such that for any $0 < s \leq \lambda_{\Delta}(D), s \neq \frac{1}{2}$, the Laplace operator $\Delta$ is surjective as an operator from $H^{1+s}(D) \cap H^1_0(D)$ to $H^{s-1}(D)$. In other words, let $\lambda_{\Delta}(D)$ be no larger than the order of the strongest singularity of the Laplace operator with homogeneous Dirichlet boundary conditions on $D$.
\end{definition}

In general, the value of $\lambda_\Delta(D)$ depends on the geometry of $D$, and the type of boundary conditions imposed. For the particular model problem \eqref{mod1}, we have that $\lambda_\Delta(D) = 1$ for convex domains. For non-convex domains, we have $\lambda_\Delta(D) = \min_{j=1}^m \pi/\theta_j$, where $\theta_j$ is the angle at corner $S_j$, and $m$ is the number of corners in $D$. Hence, $\lambda_\Delta(D) > 1/2$ for any Lipschitz polygonal domain.

In the particular case of scalar coefficients, the following result was proven in \cite{tsgu12}.

\begin{theorem}\label{h1fe} Suppose $\mathbf A = a \mathbf I_d$, for some $a:\Omega \times D \rightarrow \R$, and let Assumptions A1-A3 hold for some $0 < t \le 1$. Then,
\[
\|u(\omega,\cdot)\|_{H^{1+s}(D)} \lesssim \frac{\mathbf A_{\mathrm{max}}(\omega)\|\mathbf A (\omega, \cdot)\|^2_{C\,^{t}(\overline D, \R^{d \times d})}}{\mathbf A_{\mathrm{min}}(\omega)^4} \|f\|_{H^{t-1}(D)},
\]
for almost all $\omega \in \Omega$ and for all $0 < s < t$ such that $s \le \lambda_{\Delta}(D)$. Hence, 
\[
\|u-u_h\|_{L^p(\Omega, H^1_0(D))} \; \le  \;  C_{\mathbf A,f} \; h\,^s, \qquad \text{for all} \ p < p_* \,,
\]
with $C_{\mathbf A,f} < \infty$ a constant that depends on the input data, but 
is independent of $h$.
If A1-A3 hold with $t=\lambda_{\Delta}(D)=1$, then $\|u-u_h\|_{L^p(\Omega, H^1_0(D))} \le C_{\mathbf A,f} \, h$. 
\end{theorem}

From Theorem\ref{h1fe}, one can easily deduce convergence rates $\alpha$ and $\beta$ for Theorem 1, for $Q = |u|_{H^1(D)}$. Assume $p_*>2$, and $0 < t < 1$. Using the reverse triangle inequality, we get
\begin{equation}\label{eq:alpha}
\EE \left[ \big||u|_{H^1(D)} - |u_{h_\ell}|_{H^1(D)}\big|\right] \leq \EE \left[ |u - u_{h_\ell}|_{H^1(D)}\right] = \|u-u_{h_\ell}\|_{L^1(\Omega, H^1_0(D))} \lesssim C_{\mathbf A,f} \, h_\ell^s,
\end{equation}
and so $\alpha=s$. Similarly, using  $\VV(X) \leq \EE\left[ X^2\right]$, the reverse triangle inequality, the triangle inequality and $\EE\left[ X^2\right] = \|X\|^2_{L^2(\Omega)}$, we have
\begin{equation}\label{eq:beta}
\VV\left[ \left||u_{h_\ell}|_{H^1(D)} - |u_{h_{\ell-1}}|_{H^1(D)}\right|\right] \leq \EE\left[ \big( |u_{h_\ell} - u_{h_{\ell-1}}|_{H^1(D)} \big)^2\right] \lesssim C_{\mathbf A,f} \, h_\ell^{2s},
\end{equation}
and so $\beta=2s$. If $t=1$, one can similarly show that assumptions M1--M2 are satisfied with $\alpha=1$ and $\beta=2$.

As in the deterministic setting, one can use Theorem\ref{h1fe}, together with a duality argument, to prove convergence of the finite element error for other quantities of interest. These quantities include $\|u\|_{L^2(D)}$, for which one can prove convergence rates twice those of the $H^1(D)$--seminorm (see \cite{cst11}), and all functionals which are continuously Fr\'echet differentiable, for which one can prove convergence rates up to twice those of the $H^1(D)$--seminorm, depending on the functional (see \cite{tsgu12}).

The remainder of this section will be devoted to extending the theory above. In \S \ref{sec:aniso}, we prove an analogue of Theorem\ref{h1fe} in the case of more general tensor coefficients $\mathbf A$. In \S \ref{sec:linfty}, we prove convergence of a functional that does not fit into the framework of functionals covered in \cite{tsgu12}, namely point evaluations.

\subsection{Regularity of the Solution}
\label{sec:aniso}
The main result in this section is that Theorem\ref{h1fe} holds also for more general tensor coefficients.

\begin{theorem}\label{thm:regu} Let Assumptions A1-A3 hold for some $0 < t \le 1$. Then,
\[
\|u(\omega,\cdot)\|_{H^{1+s}(D)} \lesssim \frac{\mathbf A_{\mathrm{max}}(\omega)\|\mathbf A (\omega, \cdot)\|^2_{C\,^{t}(\overline D, \R^{d \times d})}}{\mathbf A_{\mathrm{min}}(\omega)^4} \|f\|_{H^{t-1}(D)},
\]
for almost all $\omega \in \Omega$ and for all $0 < s < t$ such that $s \le \lambda_{\Delta}(D)$. Hence, $u \in L^p(\Omega,H^{1+s}(D))$, for any $p < p_*$.
If A1-A3 hold with $t=\lambda_{\Delta}(D)=1$, then the above bound holds with $s=1$, and $u \in L^p(\Omega,H^{2}(D))$.
\end{theorem}

The proof of Theorem\ref{thm:regu} is very similar to the proof in the case of scalar coefficients, which can be found in full in \cite[\S 5]{tsgu12} and \cite[\S A]{cst11}. We will therefore only give the final result, together with the main ideas of the proof, in this section. For a detailed proof, see \cite{teckentrup_thesis}.

\begin{proof}[Proof of Theorem\ref{thm:regu}] (Sketch)
The proof follows closely that of \cite[\S 5.1]{tsgu12}. We denote by $A_\omega$ the differential operator $-\mathrm{div}(\mathbf A \nabla \cdot)$. From a result in perturbation theory, it suffices to show that there exists a constant $C_{\scriptscriptstyle \mathrm{semi}}(\omega)$ such that
\begin{equation}\label{eq:semi}
\|v\|_{H^{1+s}(D)} \leq C_{\scriptscriptstyle \mathrm{semi}}(\omega) \|A_\omega v \|_{H^{s-1}(D)},
\quad \text{for all} \ v \in H^{1+s}(D) \cap H^1_0(D),
\end{equation}
in order to conclude that $u(\omega,\cdot) \in H^{1+s}(D)$.
To prove the existence of such a constant $C_{\scriptscriptstyle \mathrm{semi}}(\omega)$, we combine regularity results for operators with constant coefficients in polygonal domains, with regularity results for operators with variable coefficients in smooth domains. 

We first choose a smooth ($C\,^2$) domain $D' \subset D$, which roughly speaking coincides with $D$ away from the corners, and does not contain any of the corners. A slight generalization of the proof in \cite[\S A]{cst11}, establishes the required result \eqref{eq:semi} for all functions $w \in H^{1+s}(D') \cap H^1_0(D')$. 

Secondly, in order to characterize the behavior of the function $v$ near the corners, we let $W$ be a polygonal subdomain of $D$, which includes some corner $S_j$. To prove \eqref{eq:semi} in $W$, we first show that 
\[
\mathbf A_{\mathrm{min}}(\omega) \, \|w\|_{H^{1+s}(W)} \lesssim \, \|A^j_\omega w \|_{H^{s-1}(W)},
\quad \text{for all} \ w \in H^{1+s}(W) \cap H^1_0(W),
\]
where $A^j_\omega$ is the operator $A_\omega$, with coefficients frozen at the corner $S_j$. This is done by using the Cauchy--Schwartz and the Poincar\'e inequalities, together with the definition of $\mathbf A_{\mathrm{min}}(\omega)$. 

Using the triangle inequality, we then have
\[
\mathbf A_{\mathrm{min}}(\omega) \, \|w\|_{H^{1+s}(W)} \;\lesssim\; \left(\|A_\omega w \|_{H^{s-1}(W)} + \|A_\omega w - A^j_\omega w\|_{H^{s-1}(W)} \right), 
\]
and so the crucial step is now to bound $\|A_\omega w - A^j_\omega w\|_{H^{s-1}(W)}$. This is done by showing that this difference can be bounded in terms of $\|\mathbf A - \mathbf A(S_j)\|_{\mathcal C(\overline W,\R^{d \times d})}$ and $\|\mathbf A\|_{\mathcal C\,^t(\overline W,\R^{d \times d})}$, which, by our regularity assumption A2, can be made arbitrarily small by making $W$ arbitrarily small. This establishes \eqref{eq:semi} for functions $w \in H^{1+s}(W) \cap H^1_0(W)$.

The final estimate \eqref{eq:semi}, can then be deduced by combining the two results with the help of suitable cut--off functions. The final result is that \eqref{eq:semi} holds with 
\[
C_{\scriptscriptstyle \mathrm{semi}}(\omega) = \frac{\mathbf A_{\mathrm{max}}(\omega)\|\mathbf A (\omega, \cdot)\|^2_{C\,^{t}(\overline D, \R^{d \times d})}}{\mathbf A_{\mathrm{min}}(\omega)^4}.
\]
\end{proof}

\subsection{Convergence of Point Evaluations}
\label{sec:linfty}
The aim of this section is to derive bounds on moments of $\|(u-u_h)(\omega,\cdot)\|_{L^\infty(D)}$ and $\|(u-u_h)(\omega,\cdot)\|_{W^{1,\infty}(D)}$. This will give us convergence rates of the finite element error for point evaluations of the pressure $u$ and the Darcy flux $- \mathbf A \nabla u$. A classical method used to derive these estimates, is the method of weighted Sobolev spaces by Nitsche. The results presented in this section are specific to continuous, linear finite elements on triangles, but extensions to higher spatial dimensions and/or higher order elements can be proved in a similar way (see e.g. \cite{ciarlet}).

The main result is the following theorem. A detailed proof can again be found in \cite{teckentrup_thesis}.

\begin{theorem} \label{thm:infty} Assume $u \in H^1_0(D) \cap C\,^{r}(\overline D)$, for some $0 < r \leq 2$. Then 
\begin{equation*}
\|(u-u_h)(\omega,\cdot)\|_{L^\infty(D)} \lesssim \frac{\mathbf A_{\mathrm{max}}(\omega)}{\mathbf A_{\mathrm{min}}(\omega)} \, h^{r} \, |\ln h| \, \|u(\omega,\cdot)\|_{C\,^{r}(\overline D)}.
\end{equation*}
If $1 < r \leq 2$, we furthermore have
\begin{equation*}
|(u-u_h)(\omega,\cdot)|_{W^{1,\infty}(D)} \lesssim \frac{\mathbf A_{\mathrm{max}}(\omega)}{\mathbf A_{\mathrm{min}}(\omega)} \, h^{r-1} \, |\ln h| \, \|u(\omega,\cdot)\|_{C\,^{r}(\overline D)}. 
\end{equation*}
\end{theorem}
\begin{proof} (Sketch) Using the method of weighted norms by Nitsche, as is done in for example \cite[\S 3.3]{ciarlet}, one can derive the quasi--optimality result
\begin{align*}
\|(u-u_h)(\omega,\cdot)\|_{L^\infty(D)} \, &+ \, h \,|(u-u_h)(\omega,\cdot)|_{W^{1,\infty}(D)} \lesssim \\ &\frac{\mathbf A_{\mathrm{max}}(\omega)}{\mathbf A_{\mathrm{min}}(\omega)} \, \inf_{v_h \in V_h} \left(\|u(\omega,\cdot) - v_h\|_{L^\infty(D)} \, + \, h |\ln h| \, |u(\omega,\cdot)-v_h|_{W^{1,\infty}(D)} \right),
\end{align*}
which holds for any $h$ sufficiently small, and where again the dependence on $\mathbf A$ has been made explicit. The claim of the proposition then follows from the best approximation result
\begin{equation*}
\inf_{v_h \in V_h} \|u(\omega,\cdot) - v_h\|_{L^\infty(D)} \, \lesssim h^{r} \, \|u(\omega,\cdot)\|_{C\,^{r}(\overline D)},
\end{equation*}
which can be found in e.g. \cite{s80}, and holds for all $0 < r \leq 2$. 
\end{proof}


In order to conclude on the convergence of moments of $\|(u-u_h)(\omega,\cdot)\|_{L^\infty(D)}$ and $|(u-u_h)(\omega,\cdot)|_{W^{1,\infty}(D)}$, it remains to prove a bound on moments of $\|u(\omega,\cdot)\|_{C\,^{r}(\overline D)}$, for some $0 < r \leq 2$. One way to achieve this is to use the Sobolev Embedding Theorem (see e.g. \cite[\S 3.1]{ciarlet}). We know from Theorem~\ref{thm:regu} that $u(\omega,\cdot) \in H^{1+s}(D)$, for some $0 < s \leq 1$, which gives the following convergence rates.

\begin{theorem}\label{thm:infty_conv} Let Assumptions A1--A3 be satisfied, for some $0 < t\leq 1$, and let $0 < s \leq t$ be such that $u \in L^p(\Omega,H^{1+s}(D))$, for all $p < p_*$. Then
\begin{align*}
\|u-u_h\|_{L^p(\Omega,L^\infty(D))} &\lesssim \, C_{\mathbf A,f} \; h\,^{1+s-d/2}, \quad \forall s \, \text{ s.t. } \, \frac{d}{2}-1 < s \leq 1, \\
\|u-u_h\|_{L^p(\Omega,W^{1,\infty}(D))} &\lesssim \, C_{\mathbf A,f} \; h\,^{s-d/2}, \qquad \forall s \, \text{ s.t. } \, \frac{d}{2} < s  \leq 1,
\end{align*}
for all $p < p_*$, with $C_{\mathbf A,f}$ a finite constant dependent on $\mathbf A$ and $f$, but independent of $h$ and $u$.
\end{theorem}
\begin{proof} This follows directly from Theorem\ref{thm:infty} and the Sobolev Embedding Theorem.
\end{proof}

Alternatively, one can use Schauder theory to derive a bound on $\|u(\omega,\cdot)\|_{C\,^{r}(\overline D)}$ directly, without going through the Sobolev Embedding Theorem. Theorems 8.33 and 8.34 in \cite{GT} give the following.

\begin{theorem}\label{thm:schauder} Let Assumptions A1--A3 be satisfied, for some $0 < t \leq 1$ and $p_*> d/(1-t)$, and suppose $D$ is a $C\,^{1+t}$ domain. Then $u(\omega,\cdot) \in C\,^{1+t}(\overline D)$, and
\begin{equation}
\|u(\omega,\cdot)\|_{C\,^{1+t}(\overline D)} \leq C_{\scriptstyle \mathrm{schauder}} \left( \|u(\omega,\cdot)\|_{C(\overline D)} + \|f(\omega,\cdot)\|_{L^{p_*}(D)}\right),
\end{equation}
where the constant $C_{\scriptstyle \mathrm{schauder}}$ depends on $\mathbf A_\mathrm{min}(\omega), \mathbf A_\mathrm{max}(\omega)$ and $\|\mathbf A (\omega, \cdot)\|_{C\,^{t}(\overline D, \R^{d \times d})}$.
\end{theorem}

A similar result can again be proved for polygonal domains $D$, taking into account the singularities which can arise near the corners. The regularity of $u$, or more precisely the number $r$ for which $u(\omega,\cdot) \in C\,^{r}(\overline D)$, will again depend on $t$ and the angles in $D$ (see e.g. \cite[\S 6]{grisvard2}).

Theorem~\ref{thm:schauder} suggests that $\|u-u_h\|_{L^p(\Omega,L^\infty(D))}$ and $\|u-u_h\|_{L^p(\Omega,W^{1,\infty}(D))}$ should converge with $h\,^{1+t}$ and $h\,^t$, respectively. These rates are better than those proved in Theorem~\ref{thm:infty_conv}, and in particular, are dimension independent. To be able to conclude rigorously on these convergence rates, we would, as in Theorem~\ref{thm:regu}, have to know exactly how the constant $C_{\scriptstyle \mathrm{schauder}}$ depends on $\mathbf A_\mathrm{min}(\omega), \mathbf A_\mathrm{max}(\omega)$ and $\|\mathbf A (\omega, \cdot)\|_{C\,^{t}(\overline D, \R^{d \times d})}$. Theorem~\ref{thm:schauder} does, however, allow us to conclude on these higher convergence rates path wise (i.e. for almost all $\omega \in \Omega$, as in Theorem~\ref{thm:infty}).

The results in this section can be used in the same way as in \eqref{eq:alpha} and \eqref{eq:beta} to prove convergence rates $\alpha$ and $\beta$ in Theorem \ref{main_thm} for point evaluations. Using the fact that $u(\omega,\cdot) \in C\, ^1(\overline D)$ (cf Theorem \ref{thm:schauder}) for almost all $\omega$, we for example have for evaluations of the norm of the Darcy flux $-\mathbf A \nabla u$ at a point $x^* \in D$ 
\[
\EE \left[ \big||\mathbf A \nabla u(x^*)| - |\mathbf A \nabla u_{h_\ell}(x^*)|\big|\right] \leq \EE \left[ |\mathbf A(x^*)| \,|(\nabla u - \nabla u_{h_\ell})(x^*)|\right] \leq \EE \left[ \mathbf A_\mathrm{max} \,|u - u_{h_\ell}|_{W^{1,\infty(D)}}\right] \lesssim C_{\mathbf A,f} \; h_\ell^\alpha,
\]
where $\alpha=s-d/2$, if we use Theorem~\ref{thm:infty_conv}, or $\alpha=t$, if we use the rates suggested by Theorem~\ref{thm:schauder}. Similarly, we have
\[
\VV\left[ \big||\mathbf A \nabla u(x^*)| - |\mathbf A \nabla u_{h_\ell}(x^*)|\big|\right] \leq  \EE \left[ \mathbf A^2_\mathrm{max} \,|u - u_{h_\ell}|^2_{W^{1,\infty(D)}}\right] \lesssim C_{\mathbf A,f} \; h_\ell^{2\alpha}, 
\]
and so $\beta=2\alpha$, where $\alpha$ is as above. This can easily be generalized to point evaluations of the Darcy flux in a given coordinate direction. The proof for point evaluations of the pressure is also similar, and leads to convergence rates $\alpha=1+s-d/2$ and $\beta=2 (1+s-d/2)$, if we use Theorem~\ref{thm:infty_conv}, and $\alpha=1+t$ and $\beta=2 (1+t)$, if we use the rates suggested by Theorem~\ref{thm:schauder}.

\section{Conclusions}
\label{sec:conc}
We have considered the application of multilevel Monte Carlo methods to elliptic PDEs with random coefficients, in the important case of coefficients which are not uniformly coercive and bounded with respect to the random parameter. This includes, for example, log--normal random fields. Under minimal assumptions on the random coefficient, we have proven convergence of the multilevel Monte Carlo algorithm, together with an upper bound on its computational cost. We have shown that the convergence analysis in \cite{tsgu12} holds also in the case of more general, tensor--valued coefficients, and also for point evaluations of the pressure and the flux.

\section*{Acknowledgments}
The author would like to thank Dr Julia Charrier, Prof Andrew Cliffe, Prof Mike Giles, Prof Robert Scheichl and Dr Elisabeth Ullmann, for their contributions to the papers on which this work is built.

\bibliographystyle{plain}
\bibliography{bibWSC}

\begin{thebibliography}{10}

\bibitem{bsz11}
A.~Barth, Ch. Schwab, and N.~Zollinger.
\newblock Multi--level {M}onte {C}arlo finite element method for elliptic
  {PDE}'s with stochastic coefficients.
\newblock {\em Numerische Mathematik}, 119(1):123--161, 2011.

\bibitem{brandt2}
A.~Brandt and V.~Ilyin.
\newblock Multilevel {M}onte {C}arlo methods for studying large scale phenomena
  in fluids.
\newblock {\em Journal of Molecular Liquids}, 105(2-3):245--248, 2003.

\bibitem{brenner_scott}
S.~C. Brenner and L.~R. Scott.
\newblock {\em The Mathematical Theory of Finite Element Methods}, volume~15 of
  {\em Texts in Applied Mathematics}.
\newblock Springer, third edition, 2008.

\bibitem{charrier}
J.~Charrier.
\newblock Strong and weak error estimates for the solutions of elliptic partial
  differential equations with random coefficients.
\newblock {\em SIAM Journal on Numerical Analysis}, 50(1):216--246, 2012.

\bibitem{cst11}
J.~Charrier, R.~Scheichl, and A.~L. Teckentrup.
\newblock Finite element error analysis of elliptic {PDE}s with random
  coefficients and its application to multilevel {M}onte {C}arlo methods.
\newblock Technical Report 02/11, University of Bath, 2011.
\newblock Available at {\tt http://www.bath.ac.uk/math-sci/bics/papers/}.

\bibitem{ciarlet}
P.~G. Ciarlet.
\newblock {\em The {F}inite {E}lement {M}ethod for {E}lliptic {P}roblems}.
\newblock North--Holland, 1978.

\bibitem{cgst11}
K.~A. Cliffe, M.~B. Giles, R.~Scheichl, and A.~L. Teckentrup.
\newblock Multilevel {M}onte {C}arlo methods and applications to elliptic
  {PDE}s with random coefficients.
\newblock {\em Computing and Visualization in Science}, 14(1):3--15, 2011.

\bibitem{demarsily}
G.~de~Marsily.
\newblock {\em Quantitative Hydrogeology}.
\newblock Academic Press, 1986.

\bibitem{demarsilyetal}
G.~de~Marsily, F.~Delay, J.~Goncalves, P.~Renard, V.~Teles, and S.~Violette.
\newblock Dealing with spatial heterogeneity.
\newblock {\em Hydrogeology Journal}, 13:161--183, 2005.

\bibitem{delhomme}
P.~Delhomme.
\newblock Spatial variability and uncertainty in groundwater flow param- eters,
  a geostatistical approach.
\newblock {\em Water Resources Research}, 15(2):269--280, 1979.

\bibitem{GT}
D.~Gilbarg and N.~S. Trudinger.
\newblock {\em Elliptic partial differential equations of second order}.
\newblock Classics in Mathematics. Springer, Berlin, 2001.

\bibitem{giles1}
M.~B. Giles.
\newblock Multilevel {M}onte {C}arlo path simulation.
\newblock {\em Operations Research}, 256:981--986, 2008.

\bibitem{gittelson}
C.~J. Gittelson.
\newblock Stochastic {G}alerkin discretization of the log-normal isotropic
  diffusion problem.
\newblock {\em Mathematical Models \& Methods in Applied Sciences},
  20(2):237--263, 2010.

\bibitem{gsu12}
I.~G. Graham, R.~Scheichl, and E.~Ullmann.
\newblock Finite element error analysis for mixed formulations of elliptic
  {PDE}s with lognormal coefficients.
\newblock In preparation, 2012.

\bibitem{grisvard2}
P.~Grisvard.
\newblock {\em Elliptic problems in non--smooth domains}.
\newblock Pitman, 1985.

\bibitem{hackbusch}
W.~Hackbusch.
\newblock {\em Elliptic differential equations}, volume~18 of {\em Springer
  Series in Computational Mathematics}.
\newblock Springer, 2010.

\bibitem{heinrich}
S.~Heinrich.
\newblock Multilevel {M}onte {C}arlo methods.
\newblock In {\em Large Scale Scientific Computing. 3rd international
  conference}, volume 2179 of {\em Lecture notes in Computer Science}, pages
  58--67. Berlin, Springer, 2001.

\bibitem{s80}
A.~Schatz.
\newblock A {W}eak {D}iscrete {M}aximum {P}rinciple and {S}tability of the
  {F}inite {E}lement {M}ethod in {$L_\infty$} on {P}lane {P}olygonal {D}omains.
  {I}.
\newblock {\em Mathematics of Computation}, 34(149):77--91, 1980.

\bibitem{teckentrup_thesis}
A.~L. Teckentrup.
\newblock {\em Multilevel {M}onte {C}arlo methods for elliptic {PDE}s with
  random coeffcients}.
\newblock PhD thesis, University of Bath, 2013.

\bibitem{tsgu12}
A.~L. Teckentrup, R.~Scheichl, M.~B. Giles, and E.~Ullmann.
\newblock Further analysis of multilevel {M}onte {C}arlo methods for elliptic
  {PDE}s with random coefficients.
\newblock Technical Report arXiv:1204.3476v1, {\tt arXiv.org}, 2012.
\newblock Available at {\tt http://arxiv.org/abs/1204.3476}.

\end{thebibliography}

\end{document}